\documentclass[12pt, reqno]{amsart}
\usepackage{enumitem}
\usepackage{amssymb}
\usepackage{tikz-network}
\usepackage{adjustbox}
\usepackage{graphicx}
\graphicspath{ {./Images/} }
\usepackage{float}
\usepackage{hyperref}
\usepackage{cleveref}
 \usepackage{tikz}
 \pgfdeclarelayer{nodelayer}
 \pgfdeclarelayer{edgelayer}
 \pgfsetlayers{edgelayer,nodelayer,main}
 \usetikzlibrary{}
 \tikzstyle{Black Vertex}=[fill=black, draw=black, shape=circle]
  \tikzstyle{none}=[fill=none, draw=none, shape=circle]
  \tikzstyle{Green Edge}=[-, fill=none, draw=green]
\tikzstyle{Red Edge}=[-, fill=none, draw=red]
\tikzstyle{Blue Edge}=[-, draw=blue]
\tikzstyle{Orange Edge}=[-, draw={rgb,255: red,255; green,128; blue,0}]
\tikzstyle{Dashed Edge}=[dashed]
\tikzstyle{Outline}=[fill=white, draw=black, shape=rectangle]
\tikzstyle{Black Edge}=[-]
\newtheorem{theorem}{Theorem}[section]
\newtheorem{lemma}[theorem]{Lemma}
\newtheorem{corollary}[theorem]{Corollary}
\newtheorem{proposition}[theorem]{Proposition}

\newtheorem{sublemma}{}[theorem]

\theoremstyle{definition}

\theoremstyle{remark}

\numberwithin{equation}{section}
\newcommand{\ba}{\backslash}

\newcommand{\si}{\rm{si}}
\newcommand{\co}{\rm{co}}

\begin{document}

\title{On $k$-connected Sets in Matroids}

\author{Sean Boros, James Oxley}
\address{Mathematics Department\\
  Louisiana State University\\
  Baton Rouge, Louisiana}
\email{sboros1@lsu.edu, oxley@math.lsu.edu}
\keywords{matroid connectivity, $k$-connected set, matroid tangle}

\subjclass[2020]{05B35}
\date{\today}

\begin{abstract}
   In 1966, Tutte defined $k$-connected matroids for each $k \ge 2$. Robertson and Seymour built on Tutte's idea to define tangles as a way to identify highly connected parts of a matroid. Motivated by tangles, a set $Z$ in a matroid $M$ has been defined to be $3$-connected if, for each $j$ in $\{1,2\}$, there is no $j$-separation $(X,Y)$ of $M$ such that each of $|Z \cap X|$ and $|Z \cap Y|$ is at least $j$. This paper solves a problem of Whittle for $k= 3$ by proving that a minor-minimal matroid having $Z$ as a $3$-connected set has at most $3|Z| - 4$ elements. We also identify the matroids for which this bound is attained.
\end{abstract}

\maketitle

\section{Introduction}
  \label{introduction}

The matroid terminology used here will follow \cite{Ox}. There has been extensive study of the properties of connected matroids, that is, of matroids that cannot be written as the direct sum of two non-empty matroids. In 1966, Tutte \cite{Tutte} defined $k$-connected matroids for all $k\geq 2$. In particular, a matroid is 2-connected if and only if it is connected. Tutte was motivated by a desire to extend vertex connectivity in graphs to matroids. His definition guaranteed that the dual of a $k$-connected matroid is also $k$-connected. 

Robertson and Seymour \cite{Rs} introduced tangles to identify highly connected parts of a graph or matroid. Geelen, Gerards, and Whittle have proved a number of important results for tangles in matroids (see, for example, \cite{Ggrw, Ggw}). In work that uses tangles as the primary tool, Brettell et al. \cite{Bjosw} defined a set $Z$ in a matroid $M$ to be \textit{$k$-connected} for an integer $k$ exceeding one provided that $\lambda(A)\geq \text{min}\{|Z\cap A|, |Z-A|, k-1\}$ for all subsets $A$ of $E(M)$. 

Using tangles, the following theorem was proved in \cite[Theorem~2]{Bjosw}. It shows that the presence of a large $k$-connected set in a matroid $M$ guarantees the presence of such a set in a minor $N$ of $M$ where $N$ is \textit{weakly $4$-connected}. This means that $N$ is 3-connected and $\lambda(X) \geq 3$ for every subset $X$ of $E(N)$ with $|X|$ and $|E(N)-X|$ exceeding four.

\begin{theorem}
    Let $Z$ be an $n$-element $k$-connected set in a matroid $M$ where $n\geq 3k-5$ and $k \geq 4$. Then $M$ has a weakly $4$-connected minor that contains an $n$-element $k$-connected set. 
\end{theorem}

In this paper, we explore the properties of $k$-connected sets when $k$ is 2 or 3. The first case can be completely described as follows by using a result of Tutte \cite[6.5]{Tutte}. 

\begin{proposition}
\label{thm2}
    The following are equivalent for a subset $Z$ of the ground set of a matroid $M$.
    \begin{enumerate}[label=(\roman*)]
        \item $Z$ is a $2$-connected set.
        \item $Z$ is contained in a connected component of $M$.
        \item $Z$ is the ground set of a $2$-connected minor of $M$.
    \end{enumerate}
\end{proposition}

In a private communication, Geoff Whittle, after noting the preceeding consequence of Tutte's work, posed the following problem. Find a sharp bound on the size of a matroid $M$ that contains a $k$-connected set $Z$ such that there is no element $x$ of $E(M)-Z$ for which $Z$ is $k$-connected in $M\ba x$ or in $M/x$. He also sought a description of the extremal examples. The next theorem, which is our main result, solves these problems when $k=3$. 

The matroids that attain equality in the bound in the theorem are obtained by attaching copies of $M(K_4)$ to a common line. More precisely, for $k\geq 2$, let $(p_1, p_2, \ldots, p_k)$ be a partition $P$ of the integer $k$, that is, $p_1+p_2+\ldots+p_j = k$ for some $j$ in $[k]$ and $p_1\geq p_2\geq \ldots \geq p_j \geq 1$, while $p_{j+1} = p_{j+2} = \ldots = p_k = 0$. We construct the matroid $N_P$ by beginning with a $(k+2)$-point line $L$ on the set $\{z_1,z_2,e_1,e_2, \ldots, e_k\}$. For each $i$ in $[k]$, let $M_i$ be a copy of $M(K_4)$ that has $\{z_1, z_2, a_i\}$ as a triangle and $\{z_1, z_2, z_{i+2}\}$ as a triad.

Let $p_0 = 0$. For each $t$ in $[j]$ and each $i$ with $p_{t-1} +1 \leq i \leq p_t$, identify $a_i$ with $e_t$ and attach $M_i$ to $L$ via generalized parallel connection across the triangle $\{z_1, z_2, e_t\}$. Finally, delete $\{e_1, e_2, \ldots, e_k\}$ from the resulting matroid to obtain $N_P$. Then $|E(N_P)| = 3k+2$ and $r(N_P) = k+2$. In particular, $|E(N_P)| = 3|Z|-4$ where $Z=\{z_1,z_2,\ldots, z_{k+2}\}$. Let $\mathcal{N}$ denote the class of all matroids that are isomorphic to $N_P$ for some partition $P$ of an integer $k$ exceeding 1. We observe that, when $P= (k,0,0,\ldots,0)$, the matroid $N_P$ is the cycle matroid of the graph that is obtained from $K_{2,k}$ by adding a new vertex $v$ that is adjacent to every other vertex. In this case, $Z$ is the set of edges meeting $v$.

\begin{theorem}
\label{thm35}
    In a matroid $M$, let $Z$ be a $3$-connected subset of $E(M)$. Assume that $M$ is minor-minimal having $Z$ as a $3$-connected set. If $|Z|\leq 3$, then $E(M)=Z$. If $|Z|\geq 4$, then $|E(M)|\leq 3|Z|-4$, and $|E(M)|= 3|Z|-4$ if and only if $M$ or $M^*$ is isomorphic to $N_P$ for some partition $P$ of an integer $k$ exceeding one where $Z=\{z_1, z_2, \ldots, z_{k+2}\}$.
\end{theorem}

\section{Preliminaries}
  \label{preliminaries}

This section presents some basic properties of $k$-connected matroids. Note that if $Z$ is a $k$-connected set in a matroid $M$, then cl$(Z)$ is not necessarily a $k$-connected set. For example, let $Z$ be the set of spokes of $M(\mathcal{W}_4)$, and let $M$ be obtained from $M(\mathcal{W}_4)$ by adding an element in parallel to one of the elements of $Z$. Then $Z$ is $3$-connected in $M$, but cl$(Z)$ is not 3-connected in $M$. Although adding elements to a $k$-connected set can destroy $k$-connectedness, removing elements cannot. In particular, for all $k\geq 2$, the empty set is $k$-connected in every matroid.
  
\begin{lemma}
    Every subset of a $k$-connected set in a matroid $M$ is $k$-connected in $M$.
\end{lemma}

\begin{proof}
    Assume $Z$ is $k$-connected in $M$. Suppose $Z^\prime \subseteq Z$. We know that
    $$\lambda(A)\geq \text{min}\{|Z\cap A|,|Z-A|,k-1\}$$
    for all $A\subseteq E(M)$. But
    $$\text{min}\{|Z\cap A|,|Z-A|,k-1\}\geq \text{min}\{|Z^\prime\cap A|,|Z^\prime-A|,k-1\}.$$
    Thus $Z^\prime$ is $k$-connected.
\end{proof}

The following is an immediate consequence of the fact that $\lambda_M(X) = \lambda_{M^*}(X)$ for all sets $X$ in a matroid $M$.

\begin{lemma}
    \label{dual}
    A set $X$ in a matroid $M$ is $k$-connected if and only if $X$ is $k$-connected in $M^*$. 
\end{lemma}

Cunningham and Edmonds \cite{Ce1, Ce2} described a decomposition of a connected matroid using repeated 2-sums. The details of this appear in \cite[Section 8.3]{Ox}. We now summarize these details. A \textit{matroid-labeled tree} $T$ is a tree with vertex set $\{M_1,M_2,\ldots,M_k\}$ for some positive integer $k$ where each $M_i$ is a matroid. Moreover, 
\begin{enumerate}[label=(\roman*)]
    \item if $M_{j_1}$ and $M_{j_2}$ are joined by an edge $e_i$ of $T$, then $E(M_{j_1})\cap E(M_{j_2}) = \{e_i\}$, and $\{e_i\}$ is not a separator of $M_{j_1}$ or $M_{j_2}$; and
    \item if $M_{j_1}$ and $M_{j_2}$ are non-adjacent, then $E(M_{j_1})\cap E(M_{j_2})$ is empty.
\end{enumerate}

    Let $e$ be an edge of a matroid-labeled tree $T$. Suppose $e$ joins vertices labeled by $N_1$ and $N_2$. Then $T/e$ denotes the matroid-labeled tree in which the vertex that results by identifying the endpoints of $e$ is labeled by $N_1\oplus_2N_2$ and all other vertex and edge labels are unchanged. A \textit{tree decomposition} of a 2-connected matroid $M$ is a matroid-labeled tree $T$ such that if $V(T) = \{M_1,M_2,\ldots,M_k\}$ and $E(T) = \{e_1, e_2, \ldots, e_{k-1}\}$, then 
\begin{enumerate}[label=(\roman*)]
    \item $E(M) = (E(M_1)\cup E(M_2) \cup \ldots \cup E(M_k))-\{e_1,e_2,\ldots, e_{k-1}\}$;
    \item $|E(M_i)|\geq 3$ for all $i$ unless $|E(M)|<3$, in which case, $k=1$ and $M_1 = M$; and
    \item $M$ is the matroid that labels the only vertex of $T/e_1,e_2,\ldots,e_{k-1}$.
\end{enumerate}

The following result \cite{Ce2, Sey} (see also \cite[Proposition 8.3.9, Theorem~8.3.10]{Ox}) notes that every 2-connected matroid can be written in terms of 2-sums of 3-connected matroids. We observe that these 3-connected matroids may be circuits or cocircuits of size less than four.

\begin{theorem}
    \label{new}
    Let $M$ be a $2$-connected matroid. Then $M$ has a tree decomposition $T$ in which every vertex is labeled by a $3$-connected matroid. Moreover, if no two adjacent vertices of $T$ are labeled by circuits and no two adjacent vertices are labeled by cocircuits, then $T$ is unique.
\end{theorem}

The tree decomposition of $M$ whose existence and uniqueness are asserted in the last theorem is called the \textit{canonical tree decomposition} for $M$. If $e$ is an edge in a tree decomposition $T$ of a 2-connected matroid $M$, then $T\ba e$ consists of two subtrees $T_1$ and $T_2$. Let $X_i$ be the set of elements of $M$ that are elements of some vertex label in $T_i$. Then $M = N_1 \oplus_2N_2$ where $N_i$ is the matroid that labels the single-vertex tree that is obtained by contracting all of the edges of $T_i$. Observe that $E(N_i) = X_i \cup e$ and that $(X_1,X_2)$ is a 2-separation of $M$. We call this the 2-separation \textit{corresponding} to the edge $e$.

Clearly, if $k\geq 3$ and $Z$ is a $k$-connected set in $M$, then $Z$ is $(k-1)$-connected in $M$. The next two lemmas give some basic properties of $k$-connected sets.

\begin{lemma}
    \label{newone}
    Let $M$ be a $k$-connected matroid. If $Z\subseteq E(M)$, then $Z$ is $k$-connected in $M$.
\end{lemma}

\begin{proof}
    Suppose that $Z$ is not $k$-connected. Then $E(M)$ has a subset $A$ for which $\lambda(A)<\text{min}\{|Z\cap A|, |Z-A|, k-1\}$. As $\lambda(A) = \lambda(E(M)-A)$, we may assume that $|Z\cap A|\leq |Z-A|$. Suppose that $|Z\cap A| \geq k-1$. Then $\lambda(A)<k-1$. But $|A|, |E(M)-A| \geq k-1$, a contradiction to the fact that $M$ is $k$-connected. Thus, for some $s<k-1$, we have $\lambda(A)<s=|Z\cap A|$. Then, as $|A|, |E(M)-A|\geq |Z\cap A|$, it follows that $M$ has an $s$-separation, a contradiction.
\end{proof}

\begin{lemma}
\label{thm79}
    Let $Z$ be the ground set of a $k$-connected minor $M/C\ba D$ of a matroid $M$. Then $Z$ is $k$-connected in $M$.
\end{lemma}

\begin{proof}
    Suppose $Z$ is not $k$-connected in $M$. Then there is a subset $A$ of $E(M)$ such that $$\lambda (A) <\text{min}\{|Z\cap A|, |Z-A|, k-1\}.$$ Now, by, for example,  \cite[Corollary~8.2.5]{Ox}, $$\lambda_{M/ C\ba D}(Z\cap A)\leq \lambda (A) <\text{min}\{|Z\cap A|, |Z-A|, k-1\}.$$ We may assume that $|Z\cap A| \leq |Z-A|$. Suppose $|Z\cap A| \geq k-1$. Then, as $M/ C\ba D$ is $k$-connected, $$k-1 \leq \lambda_{M/ C\ba D}(Z\cap A)\leq \lambda (A) <\text{min}\{|Z\cap A|, |Z-A|, k-1\} = k-1.$$ This contradiction implies that $|Z\cap A|= s$ for some $s<k-1$. Then, since $M/ C\ba D$ is $k$-connected, $$s \leq \lambda_{M/ C\ba D}(Z\cap A)\leq \lambda (A) <\text{min}\{|Z\cap A|, |Z-A|, k-1\} = s.$$ This contradiction implies that $Z$ is $k$-connected in $M$.
\end{proof}

The converse of the last result does not hold. For example, let $M=M(\mathcal{W}_4)$ and $Z$ be its set of spokes. Then, by Lemma~\ref{newone}, $Z$ is 3-connected. But, as one can easily check, $Z$ is not the ground set of a 3-connected minor of $M$. 

By Proposition~\ref{thm2}, we see that the converse of Lemma~\ref{newone} also fails. For example, suppose $k\geq 2$ and let $Z$ be the ground set of a $k$-connected matroid $N$. Let $M = N\oplus U_{1,1}$. Then, by Lemma~\ref{thm79}, $Z$ is a $k$-connected set in $M$. Clearly $M$ is not a $k$-connected matroid.

The behavior of $k$-connected sets when $k=2$ is described in Proposition~\ref{thm2}. Although this is a straightforward consequence of the results of Tutte~\cite{Tutte}, we give the proof for completeness.

\begin{proof}[Proof of Proposition~\ref{thm2}.]
    Since (i)-(iii) hold when $Z = \emptyset$, we may assume that $Z\neq \emptyset$. By Lemma~\ref{thm79}, (iii) implies (i). Now assume that (i) holds. Suppose that $Z$ is not contained in a component of $M$. Let $X$ be a component that meets $Z$. Then $Z\cap (E-X) \neq \emptyset$. As $Z$ is 2-connected, 
    $$0 = \lambda(X) \geq \text{min}\{|Z\cap X|, |Z-X|, 1\},$$
    a contradiction. Thus $Z$ is contained in a component of $M$, that is, (i) implies (ii). 

    Finally, we show that (ii) implies (iii). Assume that $Z$ is contained in a component $N$ of $M$. Then, by a theorem of Tutte~\cite[6.5]{Tutte} 
    (see also \cite[Theorem~4.3.1]{Ox}), there is a sequence $N_1,N_2, \ldots, N_s$ of 2-connected matroids such that $N_1 = N$ and $E(N_s) = Z$, while, for all $i$ in $\{2,3,\ldots,s\}$, the matroid $N_i$ is a single-element deletion or contraction of $N_{i-1}$. Thus $Z$ is the ground set of a 2-connected minor of $M$. Therefore (ii) implies (iii) and the theorem holds.
\end{proof}

 For a positive integer $j$, let $(X,Y)$ be a $j$-separation of a matroid $M$. We say that a set $Z$ \textit{exhibits} the $j$-separation $(X,Y)$ if $|Z\cap X|\geq j$ and $|Z\cap Y| \geq j$. 

 \begin{lemma}
 \label{lem3}
     Suppose a set $Z$ is $(k-1)$-connected in a matroid $M$. Then $Z$ is not $k$-connected if and only if $Z$ exhibits an exact $(k-1)$-separation of $M$.
 \end{lemma}

\begin{proof}
    Suppose $Z$ is not $k$-connected. Then there is a subset $X$ of $E(M)$ such that 
    $$\lambda(X) < \text{min}\{|Z\cap X|, |Z-X|, k-1\}.$$
    Since $Z$ is $(k-1)$-connected,
    $$\lambda(X) \geq \text{min}\{|Z\cap X|, |Z-X|, k-2\}.$$ Therefore $|Z\cap X| \geq k-1$ and $|Z-X| \geq k-1$. Thus $\lambda(X) = k-2$, so $Z$ exhibits the $(k-1)$-separation $(X,E-X)$ and this $(k-1)$-separation is exact. 

    Conversely, suppose $Z$ exhibits an exact $(k-1)$-separation $(X,Y)$ of $M$. Then $\lambda(X) = k-2$ while $|Z\cap X| \geq k-1$ and $|Z-X|\geq k-1$. Therefore $$\lambda(X) < \text{min}\{|Z\cap X|, |Z-X|, k-1\},$$ so $Z$ is not $k$-connected. 
\end{proof}

The next two results follow from Proposition~\ref{thm2} and Lemma~\ref{lem3}.

\begin{corollary}
\label{cor1}
    In a matroid $M$, a set $Z$ of size at least four is $3$-connected if and only if $Z$ is contained in a component of $M$ and $Z$ does not exhibit a $2$-separation of $M$.
\end{corollary}

\begin{corollary}
\label{lem84}
    Let $Z$ be a $2$-connected set in a matroid $M$, and suppose $|Z|\in \{1,2,3\}$. Then $M$ has a $3$-connected minor $N$ with ground set $Z$.
\end{corollary}

\begin{proof}
    Since each 2-connected matroid with at most three elements is 3-connected, the result follows immediately by Proposition~\ref{thm2}.
\end{proof}

\section{Minor-minimal matroids containing a 3-connected set.}

In this section, we prove Theorem~\ref{thm35}. We begin with a preliminary result. 

\begin{lemma}
\label{lem31}
    Suppose $M$ is a minor-minimal matroid in which $Z$ is a $3$-connected set. If $|Z|\geq 4$, then $M$ is $2$-connected.
\end{lemma}

\begin{proof}
    Since $Z$ is 3-connected, $Z$ is 2-connected. Thus, by Proposition~\ref{thm2}, $Z$ is in a component $N$ of $M$. We may assume that $N\neq M$ otherwise the result holds. By the minimality of $M$ and Proposition~\ref{thm2}, in $N$, the set $Z$ is not 3-connected, but $Z$ is 2-connected. Thus $N$ has a 2-separation $(X,Y)$ in which $|Z\cap X|$ and $|Z\cap Y|$ are both at least 2. Since $M=N\oplus (M\ba E(N))$, one easily checks that $(X,Y\cup (E(M)-E(N))$ is a 2-separation of $M$. Thus 
    $$\lambda_M(X) < \text{min}\{|Z\cap X|, |Z-X|, 2\}.$$ Hence $Z$ is not 3-connected in $M$, a contradiction.
\end{proof}

The following lemma will be used in the proof of the next theorem. 

\begin{lemma}
\label{newtwo}
    Let $Z$ be a $3$-connected set in a $2$-connected matroid $M$. Suppose $M=M_1\oplus_2M_2$ where $p$ is the basepoint of the $2$-sum.  
    \begin{enumerate}[label=(\roman*)]
        \item If $Z\subseteq E(M_1)$, then $Z$ is $3$-connected in $M_1$. 
        \item If $Z\cap E(M_2) = \{z\}$ and $M_1^\prime$ is obtained from $M_1$ by relabeling the basepoint $p$ by $z$, then $M_1^\prime$ is a minor of $M$ having $Z$ as a $3$-connected set.
    \end{enumerate}
\end{lemma}

\begin{proof}
    Since $M$ is 2-connected, $M_1$ is 2-connected. Assume that $Z\subseteq E(M_1)$. Then $Z$ is 2-connected in $M_1$. Suppose that $Z$ is not 3-connected in $M_1$. Then, by Lemma~\ref{lem3}, $Z$ exhibits an exact 2-separation $(X,Y)$ of $M_1$. Without loss of generality, the basepoint $p$ of the 2-sum $M_1\oplus_2 M_2$ lies in $Y$. Then $(X,(Y\cup E(M_2))-p)$ is an exact 2-separation of $M$ that is exhibited by $Z$, a contradiction. Thus (i) holds. 

    Now suppose that $Z\cap E(M_2) = \{z\}$. Take a circuit $C$ of $M_2$ containing $\{z,p\}$. Then $M_2\ba (E(M_2)-C)/(C-\{z,p\})$ is a 2-circuit with ground set $\{z,p\}$. It follows that $M\ba (E(M_2)-C)/(C-\{z,p\})$ is $M_1^\prime$ and that $E(M_1^\prime)$ contains $Z$. Then $M_1^\prime$ is 2-connected as $M_1^\prime \cong M_1$. Thus $Z$ is 2-connected in $M_1^\prime$. If $Z$ exhibits an exact 2-separation $(X,Y)$ of $M_1^\prime$ where $z\in Y$, then $(X,Y\cup (E(M_2)-p))$ is an exact 2-separation of $M$ that is exhibited by $Z$, a contradiction. Thus $Z$ is 3-connected in $M_1^\prime$. 
\end{proof}

\begin{theorem}
\label{thm90}
    Let $Z$ be a $3$-connected set in a $2$-connected matroid $M$ and suppose that $|Z|\geq 4$. Then $M$ has a $3$-connected minor $N$ such that $Z\subseteq E(N)$. Moreover, the canonical tree decomposition $T$ of $M$ has a vertex that is labeled by a $3$-connected matroid $M_0$ such that $N$ can be obtained from $M_0$ by relabeling each basepoint in $E(M_0)$ by an element of $N$.
\end{theorem}

\begin{proof}
    We argue by induction on $|V(T)|$. If $T$ has only one vertex, then $M$ is 3-connected and we may take $N=M$. Assume the result holds if $|V(T)|<n$, and let $|V(T)|=n$. Let $p$ be an edge of $T$. Since there is a 2-separation $(X_1,X_2)$ of $M$ corresponding to $p$, we may assume that $|X_2\cap Z|\leq 1$ by Corollary~\ref{cor1}. Then $M=M_1\oplus_2M_2$ where $E(M_i) = X_i\cup p$. Let $T_1$ be the canonical tree decomposition of $M_1$. Now, as $Z$ is 3-connected in $M$, either $Z\subseteq E(M_1)$ or $Z\cap E(M_2) = \{z\}$. Thus, by Lemma~\ref{newtwo}, $Z$ is 3-connected in $M_1$, or $Z$ is 3-connected in $M_1^\prime$ where $M_1^\prime$ is obtained from $M_1$ by relabeling $p$ by $z$. Since $|V(T_1)|<|V(T)|$, the result follows by induction.
\end{proof}

\begin{lemma}
\label{lem895}
    Let $M$ be a matroid and suppose that $Z \subseteq E(M)$, that $|Z|\geq 4$, and that $e \in E(M)-Z$. Suppose $M$ is minor-minimal having $Z$ as a $3$-connected set. 
    \begin{enumerate}[label=(\roman*)]
        \item If $\si$$(M/ e)$ is $3$-connected, then $e$ is in a triangle with two elements of $Z$. 
        \item If $\co$$(M\ba e)$ is $3$-connected, then $e$ is in a triad with two elements of $Z$. 
        \item The element $e$ is in a triangle or a triad with two elements of $Z$.
        \item The element $e$ is not on a line of $M$ or $M^*$ having more than three elements.
    \end{enumerate}
\end{lemma}

    \begin{proof}
        By Lemma~\ref{lem31} and Theorem~\ref{thm90}, $M$ is 3-connected. If si$(M/e)$ is 3-connected, then the minimality of $M$ implies that $M$ has a triangle containing two elements of $Z$. Thus (i) holds, and (ii) follows by duality. Part (iii) follows by Bixby's Lemma.

        Finally, suppose $e$ is on a line $L$ of $M$ having at least four elements. If $(X,Y)$ is a 2-separation of $M\ba e$, then, without loss of generality, $|X\cap L|\geq 2$. Thus $(X\cup e, Y)$ is a 2-separation of $M$. This contradiction and duality complete the proof of (iv).
    \end{proof}

The next theorem proves the main result for $|Z|=4$.

\begin{theorem}
\label{lem91}
    Suppose $M$ is a minor-minimal matroid having $Z$ as a $3$-connected set where $2\leq |Z|\leq 4$. Then $|E(M)-Z|\leq 2|Z|-4$. Moreover, equality holds exactly when $|Z|=2$ and $M$ is isomorphic to $U_{1,2}$ and when $|Z|=4$ and $M$ is isomorphic to $M(\mathcal{W}_4)$ or $\mathcal{W}^4$.
\end{theorem}

\begin{proof}
    If $|Z| \in \{2,3\}$, then, by Corollary~\ref{lem84}, $E(M) = Z$ and the result holds. Hence we may assume that $|Z|=4$. Thus, by Theorem~\ref{thm90}, $M$ is 3-connected. Hence $r(Z)\geq 2$. If $r(Z) = 2$, then $M|Z \cong U_{2,4}$ and the result holds since $M= M|Z$. 
    Thus we may assume that $r(Z)\geq 3$. Now, as $M$ is simple, $M|Z\cong U_{2,3}\oplus U_{1,1},\ U_{3,4},$ or $U_{4,4}$. Assume that $M$ is a counterexample to the theorem and that $Z=\{z_1,z_2,z_3,z_4\}$.

    \begin{sublemma}
    \label{sub0}
        $E-Z$ has no elements $a$ and $b$ such that $\{a,z_s,z_t\}$ is a triad and $\{b,z_s,z_t\}$ is a triangle where $\{s,t\} \subseteq \{1,2,3,4\}$. 
    \end{sublemma}

    Suppose that $E-Z$ does have such elements $a$ and $b$ and that $\{s,t\} = \{1,2\}$. Since $M$ is 3-connected and $r(M)\geq 3$, the elements $a$ and $b$ are distinct. As $\co$$(M\ba a)$ is not 3-connected, Bixby's Lemma implies that $\si$$(M/a)$ is 3-connected. Hence we may assume that $\{a,z_1\}$ is in a triangle with an element $z_3$ of $Z-\{z_1,z_2\}$. By duality, $\{b,z_1\}$ or $\{b,z_2\}$ is in a triad $T^*$ with an element $z_i$ of $Z-\{z_1,z_2\}$. Suppose $i=3$. Then, for $X=\{a,b,z_1,z_2,z_3\}$, we have $\lambda(X) = r(X) + r^*(X) - |X| \leq 3+3-5 = 1$. As $M$ is 3-connected, $|E(M)|\leq 6$, so $|E-Z|< 2|Z|-4$, a contradiction. Thus $i=4$. Since $\{a,z_1,z_3\}$ is a triangle and $\{b,z_j,z_4\}$ is a triad for some $j$ in $\{1,2\}$, we deduce, by orthogonality, that $j=2$. Thus $M$ has $(z_3,a,z_1,z_2,b,z_4)$ as a fan where $\{z_3,a,z_1\}$ is a triangle. Moreover, $r(Z)\geq 4$, otherwise $M|(Z\cup \{a,b\})$ is isomorphic to $M(K_4)$ or $\mathcal{W}^3$, in which case, $|E-Z| = 2 < 2|Z|-4$, a contradiction. 

    Now, deleting the triads $\{a,z_1,z_2\}$ and $\{z_2,b,z_4\}$ from $M$ gives a flat $F$ containing $z_3$. Since $|E-Z|\geq 4$, it follows that $|F-z_3|\geq 2$. Suppose an element $c$ of $F-z_3$ is in a triangle $T$ with two elements of $Z$. By orthogonality with the triads $\{a,z_1,z_2\}$ and $\{z_2,b,z_4\}$, none of $z_1, z_2,$ or $z_4$ is in $T$, a contradiction. Thus, by Lemma~\ref{lem895} (iii), $c$ is in a triad $T^*$ with two elements of $Z$. The triangles $\{z_1,a,z_3\}$ and $\{z_1,z_2,b\}$ imply that $z_1\not\in T^*$. Moreover, $z_2\not\in T^*$ otherwise, by orthogonality, $b\in T^*$ and we obtain the contradiction that $\{z_2,z_4\}$ is in a 4-point line in $M^*$. Thus $\{z_3,z_4\} \subseteq T^*$. Since $c$ was arbitrarily chosen in $F-z_3$, it follows that $M^*$ has a 4-point line containing $\{z_3,z_4\}$, a contradiction. We conclude that \ref{sub0} holds. 

    \begin{sublemma}
        \label{sub00}
        $E-Z$ has no distinct elements $a$ and $b$ such that $\{a,z_s,z_t\}$ is a triad and $\{b,z_u,z_v\}$ is a triangle where $s,t,u,v\in \{1,2,3,4\}$.
    \end{sublemma}

    Assume that such elements exist and take $\{s,t\} = \{1,2\}$. By \ref{sub0}, $\{u,v\} \neq \{1,2\}$. By orthogonality, $\{u,v\} = \{3,4\}$. By Lemma~\ref{lem895} (iii), every element of $E-Z$ is on the line of $M$ spanned by $\{z_3,z_4\}$ or the line of $M^*$ spanned by $\{z_1,z_2\}$. Thus, by Lemma~\ref{lem895} (iv), $|E-Z|\leq 2 <2|Z|-4$, a contradiction. Hence \ref{sub00} holds. 

    By \ref{sub00} and duality, we may assume that every element of $E-Z$ is in a triangle with two elements of $Z$.

    \begin{sublemma}
        \label{sub1}
        $r(Z)=4.$
    \end{sublemma}

    Assume that $r(Z)=3$. Then $r(M) = 3$. Suppose first that $M|Z\cong U_{2,3}\oplus U_{1,1}$ where $\{z_1,z_2,z_3\}$ is a triangle. By Lemma~\ref{lem895} (iv), we may assume that $E-Z$ contains elements $a$ and $b$ such that $\{a,z_4,z_1\}$ and $\{b,z_4,z_2\}$ are triangles. Then $M|(Z\cup \{a,b\})$ is 3-connected and so equals $M$. This is a contradiction as $|E-Z| = 2<2|Z|-4$. Thus $M\ncong U_{2,3} \oplus U_{1,1}$. 
    
    We now know that $M|Z \cong U_{3,4}$. As $|E-Z|\geq 4$, there are distinct elements $a,b,$ and $c$ of $E-Z$ such that $\{a,z_p,z_q\}$, $\{b,z_s,z_t\}$, and $\{c,z_u,z_v\}$ are triangles where $\{p,q\}$, $\{s,t\}$, and $\{u,v\}$ are distinct 2-element subsets of $\{1,2,3,4\}$. Then $M|(Z\cup\{a,b,c\})$ is 3-connected and we have a contradiction. Hence \ref{sub1} holds.

    By \ref{sub1}, $M|Z\cong U_{4,4}$ and every element of $M$ is in a triangle with two elements of $Z$. Construct an auxiliary graph $G$ with $V(G)=Z$ and $E(G) = E-Z$ where an element $e$ of $E-Z$ joins vertices $z_i$ and $z_j$ if and only if $\{e,z_i,z_j\}$ is a triangle. Then $G$ is well defined and is simple. As $M$ is 3-connected and minor-minimal having $Z$ as a 3-connected set, $G$ must be a 4-cycle. It follows that $M$ is isomorphic to $M(\mathcal{W}_4)$ or $\mathcal{W}^4$. This completes the proof of the theorem. 
\end{proof}

The proof of Theorem~\ref{thm35} will use the following results of Leo \cite[Theorem~1.1]{Leo} and Akkari and Oxley \cite[Theorem~5.1]{Ak}. Suppose $\triangle$ is a triangle in both $M(K_4)$ and in a matroid $N$. Then $P_\triangle(M(K_4),N)$ denotes the generalized parallel connection of $M(K_4)$ and $N$ across the triangle $\triangle$. 

\begin{theorem}
\label{thm93}
    Suppose that $M$ is a $3$-connected matroid with at least four elements. Let $C$ be a circuit of $M$ and $f$ be an element of $C$. If $M\ba e$ is not $3$-connected for all $e$ in $C-f$, then $C$ meets a triad of $M$.
\end{theorem}

\begin{theorem}
\label{thm8}
    Let $e$ and $f$ be distinct elements of a $3$-connected matroid $M$. Suppose that $\{e,f,a\}$ is a triangle, $\{e,f,b\}$ is a triad, and $M$ is not isomorphic to $U_{2,4}$. Then $M=P_\triangle(M(K_4),M^\prime)\ba \{e^\prime, f^\prime\}$ where $\triangle = \{a,e^\prime,f^\prime\}$, the other triangles of $M(K_4)$ are $\{a,e,f\}$, $\{b,e,e^\prime\}$, and $\{b,f,f^\prime\}$, and $M^\prime$ is obtained from $M/b$ by relabeling $e$ and $f$ as $e^\prime$ and $f^\prime$, respectively.
\end{theorem}

\begin{lemma}
    \label{newlem}
    In a $3$-connected matroid $N$, suppose $(z_2,a_1,z_1,a_2,z_3)$ is a fan in which $\{z_2,a_1,z_1\}$ is a triangle. Suppose $N$ also has $\{a_1,a_2,e\}$ and $\{e,z_2,z_3\}$ as triangles. If $|E(N)|>6$, then $e$ is not in a triad of $N$, and $N\ba e$ is $3$-connected.
\end{lemma}

\begin{proof}
    Let $X= \{z_2,a_1,z_1,a_2,z_3,e\}$. Suppose $e$ is in a triad $T^*$ of $N$. By orthogonality, $T^*\subseteq X$. Then $\lambda(X)=r(X)+r^*(X)-|X|\leq 3+3-5=1$. As $N$ is 3-connected, $|E(N)-X|\leq 1$. Say $E(N)-X = \{f\}$. As $\{a_1,z_1,a_2\}$ is a triad of $N$, it follows that $\{z_2,e,f,z_3\}$ is a line of $N$. The triad $T^*$ now gives a contradiction to orthogonality. Hence $e$ is not in a triad of $N$. If the line containing $\{e,z_2,z_3\}$ contains a fourth point, then $N\ba e$ is 3-connected. Therefore we may assume that $r(N)\geq 4$ and si$(N/e)$ is not 3-connected. Thus, as $e$ is not in a triad, Bixby's Lemma implies that $N\ba e$ is 3-connected. 
\end{proof}

\begin{figure}
    \centering
\begin{tikzpicture}[scale=0.5]
	\begin{pgfonlayer}{nodelayer}
		\node [style=Black Vertex] (0) at (-8, 0) {};
		\node [style=Black Vertex] (1) at (0, 4) {};
		\node [style=Black Vertex] (2) at (0, -4) {};
		\node [style=Black Vertex] (3) at (8.5, 4) {};
		\node [style=Black Vertex] (4) at (8, 0) {};
		\node [style=Outline] (5) at (0, -2) {};
		\node [style=Black Vertex] (6) at (4.25, 0) {};
		\node [style=Black Vertex] (7) at (-3.75, -2.25) {};
		\node [style=Black Vertex] (8) at (-3.5, 2.25) {};
		\node [style=Black Vertex] (9) at (3.25, 4) {};
		\node [style=Black Vertex] (10) at (4, -2) {};
		\node [style=Black Vertex] (11) at (4, 2) {};
		\node [style=none] (12) at (-9, 0) {$z_3$};
		\node [style=none] (13) at (9.25, 4.75) {$z_4$};
		\node [style=none] (14) at (9, 0) {$z_5$};
		\node [style=none] (15) at (0, 5) {$z_1$};
		\node [style=none] (16) at (0, -5) {$z_2$};
		\node [style=Outline] (17) at (0, 0) {};
		\node [style=Outline] (18) at (0, 2) {};
		\node [style=none] (19) at (-3.75, 3) {$a_3$};
		\node [style=none] (20) at (3.25, 5) {$a_4$};
		\node [style=none] (21) at (-4.25, -2.75) {$b_3$};
		\node [style=none] (22) at (4.75, -3) {$b_5$};
		\node [style=none] (23) at (4.25, 2.75) {$a_5$};
		\node [style=none] (24) at (5, -0.5) {$b_4$};
		\node [style=none] (25) at (5, 1.5) {};
		\node [style=none] (26) at (6, 1) {};
	\end{pgfonlayer}
	\begin{pgfonlayer}{edgelayer}
		\draw (0) to (2);
		\draw (0) to (1);
		\draw (1) to (2);
		\draw (3) to (2);
		\draw (3) to (1);
		\draw (2) to (4);
		\draw [style=Dashed Edge, in=90, out=0] (8) to (17);
		\draw [style=Dashed Edge, in=255, out=-15] (7) to (17);
		\draw [style=Dashed Edge, bend right, looseness=0.75] (9) to (18);
		\draw [style=Dashed Edge, bend right=45] (18) to (6);
		\draw [style=Dashed Edge, bend left=45, looseness=0.75] (5) to (11);
		\draw [style=Dashed Edge, bend right] (5) to (10);
		\draw (1) to (25.center);
		\draw (26.center) to (4);
	\end{pgfonlayer}
\end{tikzpicture}
    \caption{A depiction of the rank-5 matroid $N_{(1,1,1)}$.}
    \label{fig7.55}
\end{figure}
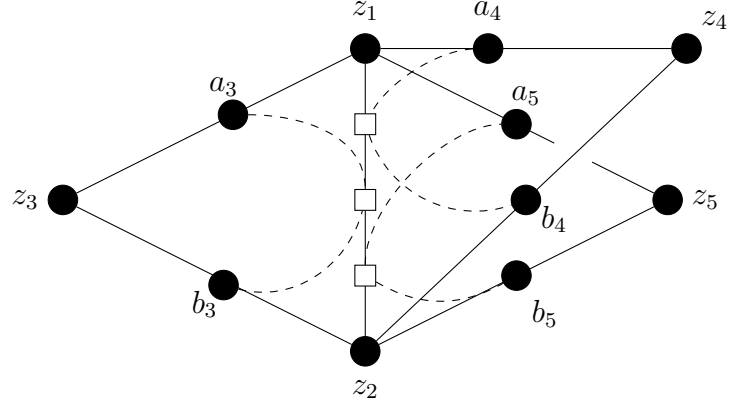

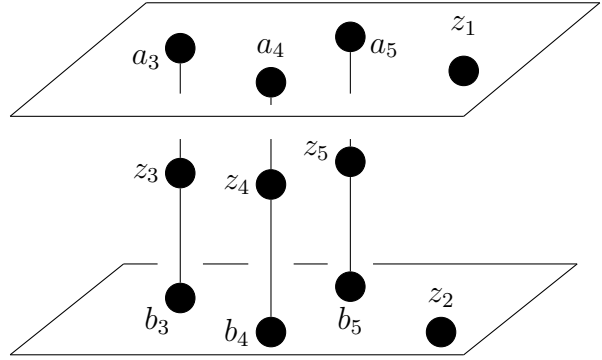
\begin{figure}
    \centering
\begin{tikzpicture}[scale=0.6]
	\begin{pgfonlayer}{nodelayer}
		\node [style=none] (0) at (-6.25, 1.5) {};
		\node [style=none] (1) at (-3.25, 4) {};
		\node [style=none] (2) at (3.75, 1.5) {};
		\node [style=none] (3) at (6.75, 4) {};
		\node [style=none] (4) at (-6.25, -3.75) {};
		\node [style=none] (5) at (-3.75, -1.75) {};
		\node [style=none] (6) at (3.75, -3.75) {};
		\node [style=none] (7) at (6.25, -1.75) {};
		\node [style=Black Vertex] (8) at (-2.5, 3) {};
		\node [style=Black Vertex] (9) at (-0.5, 2.25) {};
		\node [style=Black Vertex] (10) at (1.25, 3.25) {};
		\node [style=Black Vertex] (11) at (3.75, 2.5) {};
		\node [style=Black Vertex] (12) at (-2.5, -2.5) {};
		\node [style=Black Vertex] (13) at (-0.5, -3.25) {};
		\node [style=Black Vertex] (14) at (1.25, -2.25) {};
		\node [style=Black Vertex] (15) at (3.25, -3.25) {};
		\node [style=Black Vertex] (16) at (-2.5, 0.25) {};
		\node [style=Black Vertex] (17) at (-0.5, 0) {};
		\node [style=Black Vertex] (18) at (1.25, 0.5) {};
		\node [style=none] (19) at (3.75, 3.5) {$z_1$};
		\node [style=none] (20) at (3.25, -2.5) {$z_2$};
		\node [style=none] (21) at (-3.25, 0.25) {$z_3$};
		\node [style=none] (22) at (-1.25, 0) {$z_4$};
		\node [style=none] (23) at (0.5, 0.75) {$z_5$};
		\node [style=none] (24) at (-3.25, 2.75) {$a_3$};
		\node [style=none] (25) at (-0.5, 3) {$a_4$};
		\node [style=none] (26) at (2, 3) {$a_5$};
		\node [style=none] (27) at (-3, -3) {$b_3$};
		\node [style=none] (28) at (-1.25, -3.25) {$b_4$};
		\node [style=none] (29) at (1.25, -3) {$b_5$};
		\node [style=none] (30) at (-2.5, 2) {};
		\node [style=none] (31) at (-2.5, 1) {};
		\node [style=none] (32) at (-0.5, 1.75) {};
		\node [style=none] (33) at (-0.5, 1) {};
		\node [style=none] (34) at (1.25, 2) {};
		\node [style=none] (35) at (1.25, 1) {};
		\node [style=none] (36) at (-3, -1.75) {};
		\node [style=none] (37) at (-2, -1.75) {};
		\node [style=none] (38) at (-1, -1.75) {};
		\node [style=none] (39) at (0, -1.75) {};
		\node [style=none] (40) at (0.75, -1.75) {};
		\node [style=none] (41) at (1.75, -1.75) {};
	\end{pgfonlayer}
	\begin{pgfonlayer}{edgelayer}
		\draw (0.center) to (2.center);
		\draw (2.center) to (3.center);
		\draw (3.center) to (1.center);
		\draw (1.center) to (0.center);
		\draw (4.center) to (6.center);
		\draw (6.center) to (7.center);
		\draw (5.center) to (4.center);
		\draw (8) to (30.center);
		\draw (31.center) to (16);
		\draw (16) to (12);
		\draw (9) to (32.center);
		\draw (33.center) to (17);
		\draw (17) to (13);
		\draw (10) to (34.center);
		\draw (35.center) to (18);
		\draw (18) to (14);
		\draw (5.center) to (36.center);
		\draw (37.center) to (38.center);
		\draw (39.center) to (40.center);
		\draw (41.center) to (7.center);
	\end{pgfonlayer}
\end{tikzpicture}
    \caption{The rank-6 dual of $N_{(1,1,1)}$ where $\{a_3,a_4,a_5,z_1\}$ and $\{b_3,b_4,b_5,z_2\}$ are skew circuits.}
    \label{fig7.56}
\end{figure}

For the following proof, recall the definition of the class $\mathcal{N}$ given before the statement of Theorem~\ref{thm35}. Depictions of the matroid $N_{(1,1,1)}$ and its dual are given in Figures~\ref{fig7.55} and \ref{fig7.56}.

\begin{proof}[Proof of Theorem~\ref{thm35}]
    When $|Z|\leq 3$, the result follows from Corollary~\ref{lem84}. Moreover, the result holds when $|Z|=4$ by Theorem~\ref{lem91}. Let $M$ be a counterexample to the theorem that is minimal with respect to $|Z|$. Then $|Z|\geq 5$. Moreover, either
    \begin{enumerate}[label=(\roman*)]
        \item $|E(M)|>3|Z|-4$, or 
        \item $|E(M)|=3|Z|-4$, but $M\not\in \mathcal{N}\cup \mathcal{N}^*$.
    \end{enumerate}
    The first part of our argument will show that (i) does not hold. 

    Suppose that $Z$ spans both $M$ and $M^*$. Then 
    $$|E(M)| = r(M)+r^*(M) = r(Z)+r^*(Z) \leq 2|Z| < 3|Z|-4,$$
    a contradiction. Hence, by duality, we may assume that 

    \begin{sublemma}
        \label{sub1.4}
        $Z$ does not span $M^*$.
    \end{sublemma}

    By \ref{sub1.4}, $E-Z$ contains a circuit $K$ of $M$. Here $M\ba e$ is not 3-connected for all $e$ in $K$, otherwise $M$ is not minor minimal having $Z$ as a 3-connected set. Thus, by Theorem~\ref{thm93}, $M$ has a triad $T^*$ that meets $K$ in at least two elements. Suppose $\{a_1,a_2\} \subseteq T^*\cap K$. By Lemma~\ref{lem895} (iii) and orthogonality, each $a_i$ is in a triangle $T_{a_i}$ with two elements of $Z$. Since $|T_{a_i} \cap T^*| \neq 1$, it follows that $T^*\cap Z$ contains a single element, say $z_1$. Moreover, $z_1\in T_{a_1}\cap T_{a_2}$. Let $T_{a_1} = \{a_1,z_1,z_2\}$ and $T_{a_2} = \{a_2,z_1,z_3\}$ where $\{z_2,z_3\}\subseteq Z.$

    As co$(M\ba a_2)$ is not 3-connected, si$(M/a_2)$ is 3-connected. Now let $M_1$ be $M/a_2\ba z_1$ if $\{a_1,a_2\}$ is not in a triangle of $M$, while $M_1$ is $M/a_2\ba z_1,a_1$ if $M$ has $\{a_1,a_2\}$ in a triangle. Then $M_1$ is 3-connected since it is isomorphic to si$(M/a_2)$. Let $M_2$ be a minor of $M_1$ that is minimal having $Z-z_1$ as a 3-connected set.

    \begin{sublemma}
        \label{nom2tri}
        $M_2$ does not have a triangle containing $\{z_2,z_3\}$. 
    \end{sublemma}

    Assume $\{z_2,z_3,t\}$ is a triangle of $M_2$. Let $N=P_\triangle(M(K_4),M_2)$ where $\triangle = \{z_2,z_3,t\}$ and the other triangles of $M(K_4)$ are $T_{a_1},$ $T_{a_2},$ and $\{a_1,a_2,t\}$. Observe that, by Lemma~\ref{newlem}, $N\ba t$ is 3-connected. 

    Suppose $\{z_2,z_3,t\}$ is a triangle of $M$. Then $N=M$. Thus
    $$3|Z|-4\leq |E(M)|=3|E(M_2)|+3\leq 3|Z-z_1|-4+3 = 3|Z|-4,$$
    so $|E(M_2)| = 3|Z-z_1|-4$. Therefore $M_2\in \mathcal{N}\cup \mathcal{N}^*$. Since no member of $\mathcal{N}\cup \mathcal{N}^*$ has three collinear members of $Z$, it follows that $t\not\in Z$. Then, by Lemma~\ref{newlem}, we obtain the contradiction that $M\ba t$ is 3-connected. 

    We deduce that $\{z_2,z_3,t\}$ is not a triangle of $M$. Thus $M = N\ba t$ and \
    $$|E(M)|=|E(M_2)|+3-1 \leq 3|Z-z_1|-4+2 \leq 3|Z|-5.$$
    This contradiction completes the proof of \ref{nom2tri}.

    We now show that

    \begin{sublemma}
        \label{next}
        $|E(M)| =3|Z|-4$ and $M_2 = M\ba a_1,z_1,a_2$, so $|E(M_2)| = 3|Z-z_1|-4$. 
    \end{sublemma}

    Although $M_2$ has no triangle containing $\{z_2,z_3\}$, there is a triangle $\{z_2,z_3,t\}$ of $M_1$. If $M_2 = M_1/C\ba D$, then $t\in D$, so $t\not\in Z$. Moreover, $t\not\in E(M)$ otherwise, by Lemma~\ref{newlem}, $M\ba t$ is 3-connected. It follows that $M$ has no triangle containing $\{z_2,z_3\}$. As $M_1/C\ba (D-t)$ has $\{z_2,z_3,t\}$ as a triangle, we can obtain $M$ from $M_2$ by adjoining the elements $\{a_1,z_1,a_2\}$. Thus 
    $$3|Z|-4\leq |E(M)| = |E(M_2)|+3 \leq 3|Z-z_1| -4+3 = 3|Z|-4,$$
    that is, \ref{next} holds. 

    By the choice of $M$ and \ref{next}, $M_2\in \mathcal{N}\cup \mathcal{N}^*$. To simplify notation, for some $k\geq 2$, let $Z-z_1 = \{z^\prime, z^{\prime\prime}, x_1,x_2, \ldots, x_k\}$ and let $E(M_2)-(Z-z_1)= \{u_1,v_1,u_2,v_2, \ldots, u_k,v_k\}$. If $M_2\in \mathcal{N}$, then $Z-z_1$ spans $M_2$ and, for all $i$ in $\{1,2,\ldots,k\}$, we may assume that $M_2$ has $\{x_i,u_i,z^\prime\}$ and $\{x_i,v_i,z^{\prime\prime}\}$ as triangles and has $\{u_i,x_i,v_i\}$ as a triad. When $M_2\in \mathcal{N}^*$, we see that $M_2$ has $\{z^\prime, u_1,u_2, \ldots, u_k\}$ and $\{z^{\prime\prime}, v_1,v_2, \ldots, v_k\}$ as skew circuits whose union spans $M_2$. Moreover, $M_2$ has each $\{u_i,x_i,v_i\}$ as a triangle. 

    We know that $M$ is obtained from $M_2$ by adding elements $a_1,a_2,$ and $z_1$ where $M$ has $\{a_1,a_2,z_1\}$ as a triad, $z_1\in Z$, and $a_1,a_2\not\in Z$. Moreover, by \ref{nom2tri}, $Z-z_1$ has no elements $y_1$ and $y_2$ such that $\{z_1,a_1,y_1\}$ and $\{z_1,a_2,y_2\}$ are triangles and $\{y_1,y_2\}$ is in a triangle of $M_2$.

    \begin{sublemma}
    \label{m2n}
        $M_2\not\in \mathcal{N}$.
    \end{sublemma}

    Suppose $M_2\in \mathcal{N}$. As $M\not\in \mathcal{N}$, it follows that $\{y_1,y_2\} \neq \{z^\prime,z^{\prime\prime}\}$. By symmetry, we may assume that $(y_1,y_2)=(x_1,x_2)$. When $|Z-z_1|=4$, this case is symmetric to the case when $\{y_1,y_2\} = \{z^\prime,z^{\prime\prime}\}$. Thus we may also assume that $|Z-z_1|\geq 5$. Now let $X=\{z_1, a_1, x_1, u_1, z^\prime, u_2, x_2,\linebreak[0] a_2\}.$ Then, using the four triangles in $M|X$, it follows that this rank-4 matroid is isomorphic to $M(\mathcal{W}_4)$ or $\mathcal{W}^4$. Moreover, $M|Y\in \mathcal{N}$ where $E(M)-Y = \{z_1,a_1,a_2,x_1,u_1,v_1\}$. Since $X\cap Y = \{z^\prime,u_2,x_2\}$, it follows by a result of Oxley and Wu \cite{Wu} (see also \cite[Corollary~8.6.8]{Ox}) that $M|(X\cup Y)$ is 3-connected, that is, $M\ba v_1$ is 3-connected. This contradiction completes the proof of \ref{m2n}.

    We now know that $M_2\in \mathcal{N}^*$. By symmetry, we may assume that $y_1 = z^\prime$ or $(y_1,y_2) = (x_1,x_2)$. In the first case, since $M\ba u_i$ is not 3-connected for all $i$ and $\{u_1,u_2,\ldots,u_k,z^\prime\}$ is a circuit $D_1$, it follows by Theorem~\ref{thm93} that $M$ has a triad $T^*$ that meets $D_1$. Since $|T^*\cap D_1|\geq 2$ and $M$ has a disjoint set of triangles one containing each element of $D_1$, we obtain a contradiction to orthogonality. 

    It follows that $(y_1,y_2) = (x_1,x_2)$. Since $M\ba u_1$ and $M\ba v_1$ are not 3-connected, by Tutte's Triangle Lemma, $M$ has a triad $T^*$ that contains $u_1$. By orthogonality, $T^*$ contains an element of $D_1-u_1$ and an element of $\{v_1,x_1\}$. By orthogonality again, $T^*$ contains an element of $\{v_1,v_2,\ldots,v_k,z^{\prime\prime}\}-\{v_1\}$ or, as $x_1=y_1$, an element of $\{a_1,z_1\}$. This contradiction completes the proof of the theorem.
\end{proof}
\section*{Acknowledgements}
The authors thank Geoff Whittle for suggesting the problem considered here and for helpful initial discussions concerning it.

\end{document}